\documentclass[reqno]{amsart}

\usepackage{amssymb,amsmath,amsthm,xcolor,enumerate,hyperref,cleveref}
\usepackage[framemethod=tikz]{mdframed}

\usepackage{color}
\usepackage{fancyhdr}
\usepackage{amsxtra,ifthen}
\usepackage{verbatim}

\usepackage[numbers,sort&compress]{natbib}

\usepackage{color, hyperref}
\definecolor{blue}{rgb}{1.0,0.0,0.9}
\hypersetup{colorlinks, breaklinks,
            linkcolor=blue,urlcolor=blue,
            anchorcolor=blue, citecolor=blue}

\allowdisplaybreaks

\newtheorem{thrm}{Theorem}[section]
\newtheorem{cor}[thrm]{Corollary}
\newtheorem{lem}[thrm]{Lemma}

\theoremstyle{definition}
\newtheorem{defn}[thrm]{Definition}
\newtheorem{exm}[thrm]{Example}

\newtheorem{probl}[thrm]{Problem}

\crefrangeformat{equation}{#3(#1)#4--#5(#2)#6}

\crefname{thrm}{Theorem}{Theorems}
\crefname{lem}{Lemma}{Lemmas}
\crefname{cor}{Corollary}{Corollaries}
\crefname{prop}{Proposition}{Propositions}
\crefname{defn}{Definition}{Definitions}
\crefname{exm}{Example}{Examples}
\crefname{rem}{Remark}{Remarks}
\crefname{probl}{Problem}{Problems}
\crefname{section}{Section}{Sections}
\crefname{equation}{\unskip}{\unskip}
\crefname{enumi}{\unskip}{\unskip}

\newcommand{\af}{\alpha}
\newcommand{\bt}{\beta}

\begin{document}

\title[Lie-type Derivations of Finitary Incidence Algebras]
{Lie-type Derivations of Finitary Incidence Algebras}

\author{Mykola Khrypchenko}
\address{Departamento de Matem\'atica, Universidade Federal de Santa Catarina,  Campus Reitor Jo\~ao David Ferreira Lima, Florian\'opolis, SC, CEP: 88040--900, Brazil}
\email{nskhripchenko@gmail.com}

\author{Feng Wei}
\address{School of Mathematics and Statistics, Beijing Institute of Technology, Beijing 100081, P. R. China}
\email{daoshuo@hotmail.com, daoshuowei@gmail.com}

\begin{abstract}
Let $P$ be an arbitrary partially ordered set, $R$ a commutative ring with identity and $FI(P,R)$ the finitary incidence 
algebra of $P$ over $R$. Under some natural assumption on $R$, we prove that each Lie-type derivation of $FI(P,R)$ is 
proper, which partially generalizes the main results of \cite{ZhangKhrypchenko,WangXiao}.
\end{abstract}



\date{\today}

\maketitle

\section{Introduction and Preliminaries}\label{intro-prelim}

Let $A$ be an associative algebra over a commutative ring $R$. 
We define the {\it Lie product} $[x,y]:=xy-yx$ and {\it Jordan product} $x\circ y:=xy+yx$ for all $x,y\in A$. 
Then $(A,[\ ,\ ])$  becomes a Lie algebra and $(A,\circ)$ is a Jordan algebra. It is a fascinating topic to study the 
connection between the associative, Lie and Jordan structures on $A$. In this field, two classes of 
mappings are of crucial importance. One of them consists of mappings, preserving a type of product, 
for example, Jordan homomorphisms and Lie homomorphisms. The other one is formed by differential operators, 
satisfying a type of Leibniz formulas, such as Jordan derivations and Lie derivations.
In the AMS Hour Talk of 1961, Herstein proposed many problems concerning the structure of 
Jordan and Lie mappings in associative simple and prime rings~\cite{Her}. Roughly speaking, 
he conjectured that these mappings are all of the proper or standard form. The renowned 
Herstein's Lie-type mapping research program was formulated since then. Martindale 
gave a major force in this program under the assumption that the rings contain some 
nontrivial idempotents \cite{Mar64}. The first idempotent-free result 
on Lie-type mappings was obtained by Bre\v sar in \cite{Bre93}. We refer the reader to 
Bre\v sar's survey paper \cite{Bre04} for a detailed historical background.

We recall that an $R$-linear mapping $d: A\to A$ is called a {\it derivation} if $d(xy)=d(x)y+xd(y)$ for all $x,y\in A$, 
and it is called a {\it Lie derivation} if
$$
d([x,y])=[d(x),y]+[x,d(y)]
$$
for all $x,y\in A$.  A \textit{Lie triple derivation} is an $R$-linear mapping $d:
A\to A$ which satisfies the rule
$$
d([[x, y], z])=[[d(x), y], z]+[[x, d(y)], z]+[[x,
y], d(z)]
$$
for all $x, y, z\in A$. Obviously, a derivation is a Lie derivation, and similarly a Lie derivation is a Lie triple
derivation. But, the converse statements are not true in general. For instance, suppose that $d: A\to
A$ is a derivation and that $\tau: A\to Z(A)$ is a linear mapping from $A$ into its center $Z(A)$ such that $\tau([x,
y])=0$ for all $x, y\in A$. Then $d+\tau$ is a Lie derivation of $A$, but it is not necessarily a derivation
of $A$. Likewise, if $d: A\to A$ is a derivation and $\tau: A\to Z(A)$ is a linear mapping from
$A$ into its center $Z(A)$ such that $\tau([[x, y], z])=0$ for all $x, y, z\in A$, then $d+\tau$ is a Lie
triple derivation of $A$, but it is not necessarily a Lie derivation of $A$. In fact, if $d$ is a derivation of $A$ 
and $\tau$ is an $R$-linear mapping from $A$ into its center $Z(A)$, then $d+\tau$ is a Lie derivation (resp. Lie triple derivation) 
if and only if $\tau$ annihilates all commutators $[x,y]$ (resp. second commutators $[[x, y], z]$). 
A Lie derivation (resp. Lie triple derivation) of the form $d+\tau$, with $d$ being a
 derivation and $\tau$ a central-valued linear mapping annihilating each commutator (resp. second commutator), 
 will be said to be {\it proper}.

Taking into account the definitions of
Lie derivations and Lie triple derivations, one naturally expect to extend them in
one more general way. Suppose that $n\geq 2$ is a fixed positive
integer. Let us introduce a family of polynomials on $A$
$$
\begin{aligned}
p_1(x_1)&=x_1\\
p_2(x_1,x_2)&=[x_1,x_2]\\
p_3(x_1,x_2,x_3)&=[p_2(x_1,x_2),x_3]=[[x_1,x_2],x_3]\\
\vdots \ \ \ \ \ \ & \ \ \ \ \ \ \ \ \ \ \ \ \ \vdots \\
p_n(x_1,x_2,\cdots,x_n)&=[p_{n-1}(x_1,x_2,\cdots,x_{n-1}),x_n].
\end{aligned}
$$
The polynomial $p_n(x_1,x_2,\cdots,x_n)$ is said to be an
$(n-1)$-\textit{th commutator} ($n\geq 2$). An $R$-linear mapping $L:
A\to A$ is called a
\textit{Lie $n$-derivation} if
\begin{align}\label{Lie-n-der-formula}
	L(p_n(x_1,x_2,\dots,x_n))=\sum_{k=1}^n
	p_n(x_1,\dots,x_{k-1}, L(x_k),x_{k+1},\dots,x_n)
\end{align}
holds for all $x_1,x_2,\dots,x_n\in A$. Lie $n$-derivations were introduced by Abdullaev
\cite{Abdullaev}. In particular, he showed that every Lie $n$-derivation $L$ on a von Neumann algebra $M$ 
without central summands of type $I_1$ can be decomposed as $L=D+E$, where $D$ is an ordinary 
derivation on $M$ and $E$ is a center-valued mapping
annihilating all $n$-th commutators. This result extends an older assertion (for $n=2,3$) due to Miers \cite{Miers1, Miers2}.
By definition, a Lie derivation is a Lie $2$-derivation and a Lie triple derivation is a Lie
$3$-derivation. It is straightforward to check that every Lie $n$-derivation
on $A$ is a Lie $(n+k(n-1))$-derivation for all $k\in \Bbb{N}_0$.
Lie $2$-derivations, Lie $3$-derivations and Lie $n$-derivations are
collectively referred to as \textit{Lie-type derivations}. We shall say that a 
Lie $n$-derivation of $A$ is \textit{proper}, if $L=d+\tau$, 
where $d$ is a derivation of $A$ and $\tau$ is a central-valued linear mapping  
annihilating each \textit{$(n-1)$-th commutator} $p_n(x_1,\dots,x_n)$, where $x_1,\dots,x_n\in A$.
The question of whether each Lie-type derivation on a given algebra has the proper form  is extensively 
studied, see \cite{Benkovic1, Benkovic2, Benkovic3, BenkovicEremita1, BenkovicEremita2, Cheung, 
DuWang, FosnerWeiXiao, Lu1, Lu2, LuLiu, Mar64, MathieuVillena, Miers1, Miers2, Qi1, Qi2, Wang, WangWang, XiaoWei}. 
These works totally fulfill the Herstein's program in the background of triangular algebras and operator algebras.  
However, Cheung's work \cite{Cheung} is of special significance
in our present case. He viewed the nest algebra over a Hilbert space as a triangular algebra and 
hence circumvented  the analysis technique. This method makes it possible to study Lie 
derivations on incidence algebras in a combinatorial and linear manner \cite{ZhangKhrypchenko,WangXiao}.

Let $(P,\le)$ be a partially ordered set (poset) and $R$ a commutative ring with identity. 
With any pair of $x\le y$ from $P$ associate a symbol $e_{xy}$ and denote by 
$I(P,R)$ the $R$-module of formal sums
\begin{align}\label{formal-sum}
\alpha=\sum_{x\le y}\alpha(x,y)e_{xy},
\end{align}
where $\alpha(x,y)\in R$. If $x$ and $y$ run through a subset $X$ of the ordered 
pairs $x\le y$ in the sum \cref{formal-sum}, then it is meant that $\alpha(x,y)=0$ for any 
pair $x\le y$ which does not belong to $X$.

The sum \cref{formal-sum} is called a {\it finitary series} \cite{KhripchenkoNovikov}, 
whenever for any pair of $x,y\in P$ with $x<y$ there exists only a finite number of $u,v\in P$, 
such that $x\le u<v\le y$ and $\alpha(u,v)\ne 0$. The set of finitary series, denoted by $FI(P,R)$, 
is an $R$-submodule of $I(P,R)$ which is closed under the convolution of the series:
\begin{align}\label{conv-series}
\alpha\beta=\sum_{x\le y}\left(\sum_{x\le z\le y}\alpha(x,z)\beta(z,y)\right)e_{xy}
\end{align}
for $\alpha,\beta\in FI(P,R)$. Thus, $FI(P,R)$ is an $R$-algebra, called the {\it finitary incidence 
algebra of $P$ over $R$}. Moreover, $I(P,R)$ is a bimodule over $FI(P,R)$ by means of of \cref{conv-series}.

The incidence algebra of a poset was first considered by 
Ward in \cite{Ward} as a generalized algebra of arithmetic functions. Rota and Stanley 
developed incidence algebras as fundamental structures of enumerative combinatorial
theory and the allied areas of arithmetic function theory (see~\cite{Stanley}). Furthermore, 
Stanley~\cite{St} initiated the study of algebraic mappings and combinatorial structure of an 
incidence algebra. Since then, the automorphisms and other algebraic mappings of incidence 
algebras have been increasingly significant (see \cite{Baclawski, BruFS, BruL, Kh-aut, Kh-der, Kh-Jor, Kh-loc, Kopp,Sp, Xiao, ZhangKhrypchenko} 
and the references therein). On the other hand,
in the theory of operator algebras, the incidence algebra of a finite poset is 
referred as a bigraph algebra or a finite dimensional CSL algebra. 

The main goal of this paper is to describe Lie-type derivations of the finitary incidence algebra $FI(P,R)$.

\section{Lie $n$-derivations of $FI(P,R)$}
\label{sec-lieder-FI}

We shall identify $e_{xy}$ with $1_Re_{xy}$. We shall also write $e_x$ for $e_{xx}$.
Observe that $e_{xy}e_{uv}=\delta_{yu}e_{xv}$ by the definition of convolution. We shall also frequently use the formula
\begin{align}\label{e_x.af.e_y}
 e_x\af e_y=\af(x,y)e_{xy}
\end{align}
for all $\af\in FI(P,R)$ and $x\le y$.

\begin{lem}\label{p_n(e_x.af.e_y.etc)}
	Let $\af\in FI(P,R)$. Then for all $x<y$
	\begin{align}
		p_n(e_x,\af,e_y,\dots,e_y)&=\af(x,y)e_{xy},\ n\ge 3,\label{p_n(e_x.af.e_y...)(xy)=af(xy)}\\
		p_n(\af,e_y,\dots,e_y)(x,y)&=\af(x,y),\ n\ge 2,\label{p_n(af.e_y...)(xy)=af(xy)}\\
		p_n(\af,e_{xy},e_y,\dots,e_y)&=(\af-\af(y,y))e_{xy},\ n\ge 3.\label{p_n(af.e_xy.e_y...e_y)=(af-af(yy))e_xy}
	\end{align}
\end{lem}
\begin{proof}
	For \cref{p_n(e_x.af.e_y...)(xy)=af(xy)}, we should remark that
	\begin{align}\label{[e_x.bt_e_y]=e_x.bt.e_y}
		[e_x\bt,e_y]=e_x\bt e_y
	\end{align}
	for any $\bt\in FI(P,R)$, which is due to the fact that $e_ye_x=0$. Since $e_xe_y=e_y\bt e_x=0$, we further 
	get $[\bt e_x,e_y]=0$. And hence  
	\begin{align*}
		p_3(e_x,\af,e_y)=[[e_x,\af],e_y]=[e_x\af,e_y]-[\af e_x,e_y]=e_x\af e_y.
	\end{align*}
	It follows from \cref{[e_x.bt_e_y]=e_x.bt.e_y} that 
	\begin{align*}
		p_4(e_x,\af,e_y,e_y)=[p_3(e_x,\af,e_y),e_y]=[e_x\af e_y,e_y]=e_x\af e_y\cdot e_y=e_x\af e_y.
	\end{align*}
	By a trivial induction argument, the left-hand side of \cref{p_n(e_x.af.e_y...)(xy)=af(xy)} coincides with $e_x\af e_y$. It remains to apply \cref{e_x.af.e_y}.
	
	Now, to prove \cref{p_n(af.e_y...)(xy)=af(xy)}, let us write
	\begin{align*}
		p_2(\af,e_y)(x,y)=[\af,e_y](x,y)=(\af e_y-e_y\af)(x,y)=\af(x,y),
	\end{align*}
	whence \cref{p_n(af.e_y...)(xy)=af(xy)} by induction.
	
	Finally, let us calculate \cref{p_n(af.e_xy.e_y...e_y)=(af-af(yy))e_xy} for $n=3$:
	\begin{align*}
		p_3(\af,e_{xy},e_y)&=(\af e_{xy}-e_{xy}\af)e_y-e_y(\af e_{xy}-e_{xy}\af)\\
		&=\af e_{xy}-\af(y,y)e_{xy}\\
		&=(\af-\af(y,y))e_{xy}.
	\end{align*}
	Now since $(\af-\af(y,y))e_{xy}\cdot e_y=(\af-\af(y,y))e_{xy}$ and $e_y\cdot (\af-\af(y,y))e_{xy}=0$, we obtain \cref{p_n(af.e_xy.e_y...e_y)=(af-af(yy))e_xy} by induction for all $n\ge 3$.
\end{proof}

We shall need generalizations of Lemmas 2.3 and 3.4 from \cite{ZhangKhrypchenko}. To this end, 
we shall slightly change the definition of the restriction used in \cite{ZhangKhrypchenko}, as we did in \cite{Kh-loc}.

\begin{defn}\label{af|_x^y-defn}
For any $\af\in FI(P,R)$ and $x\le y$, we define {\it the restriction of $\af$} to $[x,y]=\{z\in X\mid x\le z\le y\}$ to be
\begin{align}\label{alpha|_x^y}
	\af|_x^y=\af(x,y)e_{xy}+\sum_{x\le v<y}\af(x,v)e_{xv}+\sum_{x<u\le y}\af(u,y)e_{uy}.
\end{align}
\end{defn}

Clearly, the mapping $\af\mapsto\af|_x^y$ is linear. Moreover, we have the next lemma, whose proof is straightforward.
\begin{lem}\label{properties-af|_x^y}
For any $\af\in FI(P,R)$, we have
\begin{enumerate}
	\item $(\alpha|_x^y)|_x^y=\alpha|_x^y$;\label{double-restr=restr}
	\item $(\af\bt)(x,y)=(\af|_x^y\bt)(x,y)=(\af\bt|_{x}^{y})(x,y)$.\label{(af.bt)(xy)=(af|_x^y.bt)(xy)=(af.bt|_x^y)(xy)}
\end{enumerate}
\end{lem}

\begin{lem}\label{L(af)(xy)=L(restr-of-af)(xy)}
Let $L$ be a Lie $n$-derivation of $FI(P,R)$, where $n\ge 2$, and $x<y$. Then
\begin{align}\label{L(af)(xy)=L(af|_x^y)(xy)}
 L(\af)(x,y)=L(\af|_x^y)(x,y).
\end{align}
\end{lem}
\begin{proof}
	It suffices to verify the case $n\ge 3$, since each Lie derivation is a Lie triple derivation. 
	We shall use \cref{p_n(e_x.af.e_y...)(xy)=af(xy),p_n(af.e_y...)(xy)=af(xy),Lie-n-der-formula}:
	\begin{align}
	L(\af)(x,y)&=p_n(e_x,L(\af),e_y,\dots,e_y)(x,y)\notag\\
	&=L(p_n(e_x,\af,e_y,\dots,e_y))(x,y)-p_n(L(e_x),\af,e_y,\dots,e_y)(x,y)\notag\\
	&\quad-p_n(e_x,\af,L(e_y),e_y,\dots e_y)(x,y)\notag\\
	&\quad-\sum_{i=1}^{n-3} p_n(e_x,\af,\underbrace{e_y,\dots,e_y}_{i},L(e_y),e_y,\dots,e_y)(x,y)\notag\\
	&=\af(x,y)L(e_{xy})(x,y)-p_2(L(e_x),\af)(x,y)\notag\\
	&\quad-p_3(e_x,\af,L(e_y))(x,y)-(n-3)p_2(\af(x,y)e_{xy},L(e_y))(x,y).\label{L(af)(xy)-expanded}
	\end{align}
	Taking into account the fact that $\af(x,y)=\af|_x^y(x,y)$ together with \cref{properties-af|_x^y}\cref{(af.bt)(xy)=(af|_x^y.bt)(xy)=(af.bt|_x^y)(xy)}, 
	we arrive at
	\begin{align*}
		p_2(L(e_x),\af)(x,y)&=(L(e_x)\af-\af L(e_x))(x,y)\\
		&=(L(e_x)\af|_x^y-\af|_x^y L(e_x))(x,y)\\
		&=p_2(L(e_x),\af|_x^y)(x,y).
	\end{align*}
	Furthermore, using \cref{properties-af|_x^y}\cref{(af.bt)(xy)=(af|_x^y.bt)(xy)=(af.bt|_x^y)(xy)} once again, we have
	\begin{align*}
		p_3(e_x,\af,L(e_y))(x,y)&=((e_x\af-\af e_x)L(e_y)-L(e_y)(e_x\af-\af e_x))(x,y)\\
		&=(\af L(e_y))(x,y)-\af(x,x)L(e_y)(x,y)-L(e_y)(x,x)\af(x,y)\\
		&=(\af|_x^y L(e_y))(x,y)-\af|_x^y(x,x)L(e_y)(x,y)-L(e_y)(x,x)\af|_x^y(x,y)\\
		&=p_3(e_x,\af|_x^y,L(e_y))(x,y).
	\end{align*}
	Thus, we see that one can replace $\af$ by $\af|_x^y$ in the right-hand side of \cref{L(af)(xy)-expanded}, whence \cref{L(af)(xy)=L(af|_x^y)(xy)}.	
	
\end{proof}

\begin{lem}\label{L(e_xy)(uy)-is-L(e_xx)(uy)}
	Let $L$ be a Lie $n$-derivation of $FI(P,R)$, where $n\ge 2$, and $u<x<y<v$. Then
	\begin{align}
	L(e_{xy})(x,v)&=L(e_y)(y,v),\label{L(e_xy)(xv)=L(e_y)(yv)}\\
	L(e_{xy})(u,y)&=L(e_x)(u,x).\label{L(e_xy)(uy)=L(e_x)(ux)}
	\end{align}
\end{lem}
\begin{proof}
	For \cref{L(e_xy)(xv)=L(e_y)(yv)}, apply $L$ to the equality 
	\begin{align}\label{e_xy=p_n(e_xy.e_y...e_y)}
		e_{xy}=p_n(e_{xy},e_y,\dots,e_y)
	\end{align}
	and use \cref{Lie-n-der-formula} to get
	\begin{align}
		L(e_{xy})&=L(p_n(e_{xy},e_y,\dots,e_y))\notag\\
		&=p_n(L(e_{xy}),e_y,\dots,e_y)+p_n(e_{xy},L(e_y),e_y,\dots,e_y)\notag\\
		&\quad+\sum_{i=0}^{n-3}p_n(e_{xy},e_y,\dots,e_y,L(e_y),\underbrace{e_y,\dots,e_y}_i)\notag\\
		&=p_n(L(e_{xy}),e_y,\dots,e_y)+\sum_{i=0}^{n-2}p_{i+2}(e_{xy},L(e_y),\underbrace{e_y,\dots,e_y}_i).\label{L(e_xy)=p_n(L(e_xy).e_y...e_y)+sum-p_(i+2)(e_xy.L(e_y).e_y...e_y)}
	\end{align}
	Observe that $[\bt,e_y](x,v)=0$ for an arbitrary $\bt\in FI(P,R)$, since $y\not\in\{x,v\}$. Hence, taking the values of the both sides of \cref{L(e_xy)=p_n(L(e_xy).e_y...e_y)+sum-p_(i+2)(e_xy.L(e_y).e_y...e_y)} at $(x,v)$, we obtain
	\begin{align*}
		L(e_{xy})(x,v)=p_2(e_{xy},L(e_y))(x,v)=(e_{xy}L(e_y)-L(e_y)e_{xy})(x,v)=L(e_y)(y,v),
	\end{align*}
	proving \cref{L(e_xy)(xv)=L(e_y)(yv)}.
	
	By invoking $e_{xy}=(-1)^{n+1} p_n(e_{xy}, e_x,\dots,e_x)$, one can show \cref{L(e_xy)(uy)=L(e_x)(ux)} in the same way as \cref{L(e_xy)(xv)=L(e_y)(yv)}. 
\end{proof}

\begin{lem}\label{L(af)_xx=L(af)_yy}
	Let $L$ be a Lie $n$-derivation of $FI(P,R)$ and $\af\in FI(P,R)$. If $R$ is $(n-1)$-torsion free, then 
	\begin{align}\label{L(af)(xx)=L(af)(yy)}
		L(\af)(x,x)=L(\af)(y,y)
	\end{align}
	for all $x<y$.
\end{lem}
\begin{proof}
	Applying $L$ to $p_n(\af,e_{xy},e_y,\dots,e_y)$ and using \cref{Lie-n-der-formula}, we have 
	\begin{align}
	L\left(p_n(\af,e_{xy},e_y,\dots,e_y)\right)&=p_n(L(\af),e_{xy},e_y,\dots,e_y)+p_n(\af,L(e_{xy}),e_y,\dots,e_y)\notag\\
	&\quad+p_n(\af,e_{xy},L(e_y),e_y,\dots,e_y)\label{p_n(af.e_xy.L(e_y).e_y...e_y)}\\
	&\quad+\sum_{i=1}^{n-3}p_n(\af,e_{xy},\underbrace{e_y,\dots,e_y}_i,L(e_y),e_y,\dots,e_y),\label{L([e_xy.af])(xy)=(L(e_xy)af)(xy)-(af.L(e_xy))(xy)+L(af)(yy)-L(af)(xx)}
	\end{align}
	where the summands \cref{L([e_xy.af])(xy)=(L(e_xy)af)(xy)-(af.L(e_xy))(xy)+L(af)(yy)-L(af)(xx),p_n(af.e_xy.L(e_y).e_y...e_y)} do not appear for $n=2$.
	
	By \cref{e_xy=p_n(e_xy.e_y...e_y),p_n(af.e_y...)(xy)=af(xy)}, we get
	\begin{align*}
		p_n(L(\af),e_{xy},e_y,\dots,e_y)(x,y)&=p_2(L(\af),e_{xy})(x,y)\\
		&=(L(\af) e_{xy}-e_{xy}L(\af))(x,y)\\
		&=L(\af)(x,x)-L(\af)(y,y).
	\end{align*}
	
	Furthermore, using \cref{p_n(af.e_y...)(xy)=af(xy)}, we have
	\begin{align*}
		p_n(\af,L(e_{xy}),e_y,\dots,e_y)(x,y)=p_2(\af,L(e_{xy}))(x,y)=(\af L(e_{xy})-L(e_{xy})\af)(x,y).
	\end{align*}
	In view of \cref{L(af)(xy)=L(af|_x^y)(xy)}, we know that 
	\begin{align*}
	\left(\af L(e_{xy})\right)(x,y)&=\sum_{x\le z\le y}\af(x,z)L(e_{xy})(z,y)\notag\\
	&=\sum_{x\le z\le y}\af(x,z)L(e_{xy}|_z^y)(z,y)\notag\\
	&=\af(x,x)L(e_{xy})(x,y),
	\end{align*}
	which is due to $e_{xy}|_z^y=\delta_{xz}e_{xy}$ for $x\le z\le y$. In an analogous manner, one can show 
	\begin{align*}
	\left(L(e_{xy})\af\right)(x,y)=\af(y,y)L(e_{xy})(x,y).
	\end{align*}
	
         Considering \cref{p_n(af.e_y...)(xy)=af(xy)}, we have
	\begin{align*}
		p_n(\af,e_{xy},L(e_y),e_y,\dots,e_y)(x,y)&=p_3(\af,e_{xy},L(e_y))(x,y)\\
		&=((\af e_{xy}-e_{xy}\af)L(e_y)-L(e_y)(\af e_{xy}-e_{xy}\af))(x,y)\\
		&=\af(x,x)L(e_y)(y,y)-\af(y,y)L(e_y)(y,y)\\
		&\quad-L(e_y)(x,x)\af(x,x)+L(e_y)(x,x)\af(y,y)\\
		&=(\af(x,x)-\af(y,y))(L(e_y)(y,y)-L(e_y)(x,x)).
	\end{align*}
	
	Finally, by invoking \cref{p_n(af.e_y...)(xy)=af(xy),p_n(af.e_xy.e_y...e_y)=(af-af(yy))e_xy} 
	we assert that 
	\begin{align*}
		p_n(\af,e_{xy},\underbrace{e_y,\dots,e_y}_i,L(e_y),e_y,\dots,e_y)(x,y)&=p_{i+3}(\af,e_{xy},\underbrace{e_y,\dots,e_y}_i,L(e_y))(x,y)\\
		&=p_2((\af-\af(y,y))e_{xy},L(e_y))(x,y),
	\end{align*}  
	the latter being
	\begin{align*}
		&((\af-\af(y,y))e_{xy}L(e_y)-L(e_y)(\af-\af(y,y))e_{xy})(x,y)\\
		&\quad=(\af(x,x)-\af(y,y))L(e_y)(y,y)-L(e_y)(x,x)(\af(x,x)-\af(y,y))\\
		&\quad=(\af(x,x)-\af(y,y))(L(e_y)(y,y)-L(e_y)(x,x)).
	\end{align*}
	
	 On the other hand, by \cref{alpha|_x^y,L(af)(xy)=L(af|_x^y)(xy),p_n(af.e_xy.e_y...e_y)=(af-af(yy))e_xy}, we 
	 arrive at
	\begin{align*}
		L\left(p_n(\af,e_{xy},e_y,\dots,e_y)\right)(x,y)&=L\left((\af-\af(y,y))e_{xy}\right)(x,y)\\
		&=L\left(((\af-\af(y,y))e_{xy})|_x^y\right)(x,y)\\
		&=L\left(\af(x,x)e_{xy}-\af(y,y)e_{xy}\right)(x,y)\\
		&=(\af(x,x)-\af(y,y))L\left(e_{xy}\right)(x,y).
	\end{align*}
	
	Thus, taking the values of the both sides of \cref{L([e_xy.af])(xy)=(L(e_xy)af)(xy)-(af.L(e_xy))(xy)+L(af)(yy)-L(af)(xx)} at $(x,y)$, we obtain
	\begin{align*}
		(\af(x,x)-\af(y,y))L\left(e_{xy}\right)(x,y)&=L(\af)(x,x)-L(\af)(y,y)\\
		&\quad+(\af(x,x)-\af(y,y))L(e_{xy})(x,y)\\
		&\quad+(n-2)(\af(x,x)-\af(y,y))(L(e_y)(y,y)-L(e_y)(x,x)),
	\end{align*}
	whence
	\begin{align}\label{(af(xx)-af(yy))(1+(n-3)(L(e_y)(yy)-L(e_y)(xx)))=0}
		L(\af)(x,x)-L(\af)(y,y)=(n-2)(\af(x,x)-\af(y,y))(L(e_y)(x,x)-L(e_y)(y,y)).
	\end{align}
	If $n=2$, then we immediately obtain \cref{L(af)(xx)=L(af)(yy)}. If $n\ge 3$, then
	taking $\af=e_y$ in \cref{(af(xx)-af(yy))(1+(n-3)(L(e_y)(yy)-L(e_y)(xx)))=0}, we get
	\begin{align*}
		L(e_y)(x,x)-L(e_y)(y,y)=-(n-2)(L(e_y)(x,x)-L(e_y)(y,y)),
	\end{align*}
	i.e.
	\begin{align*}
	(n-1)(L(e_y)(x,x)-L(e_y)(y,y))=0.
	\end{align*}
	If $R$ is $(n-1)$-torsion free, the latter yields $L(e_y)(x,x)-L(e_y)(y,y)=0$, proving thus \cref{L(af)(xx)=L(af)(yy)}.
\end{proof}

\begin{defn}\label{af_D-defn}
	Given $\af\in FI(P,R)$, define the \textit{diagonal} of $\af$ to be
	\begin{align*}
		\af_D=\sum_{x\in P}\af(x,x)e_{xx}\in FI(P,R).
	\end{align*}
	An element $\af\in FI(P,R)$ is said to be \textit{diagonal} whenever $\af=\af_D$.
\end{defn}

Recall from~\cite{SpDo} that the center $Z(FI(P,R))$ of $FI(P,R)$ consists of diagonal elements $\af\in FI(P,R)$, such that $\af(x,x)=\af(y,y)$ for all $x<y$ in $P$.
 
\begin{cor}\label{L(af)_D-belons-to-the-center}
	Let $L$ be a Lie $n$-derivation of $FI(P,R)$ and $\af\in FI(P,R)$. If $R$ is $(n-1)$-torsion free, then $L(\af)_D\in Z(FI(P,R))$.
\end{cor}

\begin{lem}\label{L(e_x)-and-L(e_xy)}
	Let $L$ be a Lie $n$-derivation of $FI(P,R)$. If $R$ is $(n-1)$-torsion free, then for all $x<y<z$:
	\begin{align}
		L(e_x)(x,y)+L(e_y)(x,y)&=0,\label{L(e_x)(xy)+L(e_y)(xy)=0}\\
		L(e_{xy})(x,y)+L(e_{yz})(y,z)&=L(e_{xz})(x,z).\label{L(e_xy)(xy)+L(e_yz)(yz)=L(e_xz)(xz)}
	\end{align}
\end{lem}
\begin{proof}
Since $e_xe_y=e_ye_x=0$, we see that $p_n(e_x,e_y,\dots,e_y)=0$. We therefore have by \cref{Lie-n-der-formula}
\begin{align}
	0&=p_n(L(e_x),e_y,\dots,e_y)+p_n(e_x,L(e_y),e_y\dots,e_y)\notag\\
	&\quad+\sum_{i=1}^{n-2}p_n(e_x,\underbrace{e_y\dots,e_y}_i,L(e_y),e_y\dots,e_y)\notag\\
	&=p_n(L(e_x),e_y,\dots,e_y)+p_n(e_x,L(e_y),e_y\dots,e_y).\label{p_n(L(e_x).e_y...e_y)+p_n(e_x.L(e_y)...e_y)}
\end{align}
By \cref{p_n(af.e_y...)(xy)=af(xy),p_n(e_x.af.e_y...)(xy)=af(xy)} the value of \cref{p_n(L(e_x).e_y...e_y)+p_n(e_x.L(e_y)...e_y)} at $(x,y)$ equals
$L(e_x)(x,y)+L(e_y)(x,y)$, whence \cref{L(e_x)(xy)+L(e_y)(xy)=0}.

To prove \cref{L(e_xy)(xy)+L(e_yz)(yz)=L(e_xz)(xz)}, observe that
\begin{align}\label{e_xz=p_n(e_xy.e_yz.e_z...e_z)}
	e_{xz}=p_{n-1}(e_{xz},e_z\dots,e_z)=p_n(e_{xy},e_{yz},e_z\dots,e_z).
\end{align}
Thus
\begin{align}
	L(e_{xz})&=L(p_n(e_{xy},e_{yz},e_z\dots,e_z))\notag\\
	&=p_n(L(e_{xy}),e_{yz},e_z\dots,e_z)+p_n(e_{xy},L(e_{yz}),e_z\dots,e_z)\notag\\
	&\quad+\sum_{i=0}^{n-3}p_n(e_{xy},e_{yz},\underbrace{e_z\dots,e_z}_i,L(e_z),e_z\dots,e_z).\label{L(e_xz)=p_n(L(e_xy).e_xz.e_z...)+p_n(e_xy.L(e_xz).e_z...)+sum}
\end{align}
It follows from \cref{p_n(af.e_y...)(xy)=af(xy)} that
\begin{align*}
	p_n(L(e_{xy}),e_{yz},e_z\dots,e_z)(x,z)&=p_2(L(e_{xy}),e_{yz})(x,z)\\
	&=(L(e_{xy})e_{yz}-e_{yz}L(e_{xy}))(x,z)\\
	&=L(e_{xy})(x,y).
\end{align*}
Similarly, we get
\begin{align*}
p_n(e_{xy},L(e_{yz}),e_z\dots,e_z)(x,z)&=p_2(e_{xy},L(e_{yz}))(x,z)\\
&=(e_{xy}L(e_{yz})-L(e_{yz})e_{xy})(x,z)\\
&=L(e_{yz})(y,z).
\end{align*}
Now, using \cref{e_xz=p_n(e_xy.e_yz.e_z...e_z),e_xy=p_n(e_xy.e_y...e_y),p_n(af.e_y...)(xy)=af(xy)}, we obtain
\begin{align*}
	p_n(e_{xy},e_{yz},\underbrace{e_z\dots,e_z}_i,L(e_z),e_z\dots,e_z)(x,z)&=p_2(e_{xz},L(e_z))(x,z)\\
	&=(e_{xz}L(e_z)-L(e_z)e_{xz})(x,z)\\
	&=L(e_z)(x,x)-L(e_z)(z,z).
\end{align*}
The latter is zero by \cref{L(af)_xx=L(af)_yy}. Thus, \cref{L(e_xy)(xy)+L(e_yz)(yz)=L(e_xz)(xz)} is proved by taking the values of the both sides of \cref{L(e_xz)=p_n(L(e_xy).e_xz.e_z...)+p_n(e_xy.L(e_xz).e_z...)+sum} at $(x,z)$.
\end{proof}

\begin{defn}\label{tau-and-d-defn}
	Let $L$ be a Lie $n$-derivation of $FI(P,R)$ and $\af\in FI(P,R)$. Define
	\begin{align}
	\tau(\af)&=L(\af)_D,\label{tau(af)=L(af)_D}\\
	d(\af)&=L(\af)-L(\af)_D.\label{d=L-tau}
	\end{align}
\end{defn}

\begin{lem}\label{tau-centr-map}
	Let $L$ be a Lie $n$-derivation of $FI(P,R)$ and $\tau$ be given by \cref{tau(af)=L(af)_D}. If $R$ is $(n-1)$-torsion free, then $\tau$ is a central-valued linear mapping which annihilates all the $n$-th commutators in $FI(P,R)$.
\end{lem}
\begin{proof}
	We have already seen in \cref{L(af)_D-belons-to-the-center} that $L(\af)_D\in Z(FI(P,R))$, so $\tau$ is center-valued.
	Since $L$ is linear, then $\tau$ is also linear. Moreover, for any $\af_1,\dots,\af_n\in FI(P,R)$, one has by \cref{Lie-n-der-formula,tau(af)=L(af)_D} and the easy fact that $(\af\bt)_D=\af_D\bt_D$:
	\begin{align*}
	\tau(p_n(\af_1,\dots,\af_n))&=L(p_n(\af_1,\dots,\af_n))_D\\
	&=\left(\sum_{i=1}^n p_n(\af_1,\dots,L(\af_i),\dots,\af_n)\right)_D\\
	&=\sum_{i=1}^n p_n((\af_1)_D,\dots,L(\af_i)_D,\dots,(\af_n)_D),
	\end{align*}
	which is zero, as diagonal elements commute.
\end{proof}

\begin{lem}\label{d-derivation}
	If $R$ is $(n-1)$-torsion free, then $d$ is a derivation of $FI(P,R)$.
\end{lem}
\begin{proof}
	We need to prove that 
	\begin{align}\label{d(af.bt)=d(af).bt+af.d(bt)}
	d(\af\bt)=d(\af)\bt+\af d(\bt)
	\end{align}
	for all $\af,\bt\in FI(P,R)$. Observe that 
	\begin{align}\label{d(af)_D-is-zero}
	d(\af)_D=L(\af)_D-L(\af)_D=0.
	\end{align}
	So \cref{d(af.bt)=d(af).bt+af.d(bt)} trivially holds at $(x,x)$ for all $x\in P$. Now, given $x<y$, 
	by \cref{d=L-tau,tau(af)=L(af)_D,L(af)(xy)=L(af|_x^y)(xy),d(af)_D-is-zero} we have
	\begin{align}
	(d(\af)\bt)(x,y)&=\sum_{x\le z\le y}d(\af)(x,z)\bt(z,y)\notag\\
	&=\sum_{x< z\le y}L(\af)(x,z)\bt(z,y)=\sum_{x< z\le y}L(\af|_x^z)(x,z)\bt(z,y)\notag\\
	&=\sum_{x< z\le y}\Big(\alpha(x,z)L(e_{xz})(x,z)+\sum_{x\le v<z}\alpha(x,v)L(e_{xv})(x,z)\notag\\
	&\quad+\sum_{x<u\le z}\alpha(u,z)L(e_{uz})(x,z)\Big)\bt(z,y).\label{(d(af).bt)(xy)=sum}
	\end{align}
	Similarly,
	\begin{align}
	(\af d(\bt))(x,y)&=\sum_{x\le z\le y}\af(x,z)d(\bt)(z,y)\notag\\
	&=\sum_{x\le z< y}\af(x,z)L(\bt)(z,y)=\sum_{x\le z< y}\af(x,z) L(\bt|_z^y)(z,y)\notag\\
	&=\sum_{x\le z< y}\af(x,z)\Big(\bt(z,y)L(e_{zy})(z,y)+\sum_{z\le v<y}\bt(z,v)L(e_{zv})(z,y)\notag\\
	&\quad+\sum_{z<u\le y}\bt(u,y)L(e_{uy})(z,y)\Big).\label{(af.d(bt))(xy)=sum}
	\end{align}
	Hence, adding \cref{(d(af).bt)(xy)=sum,(af.d(bt))(xy)=sum}, we arrive at
	\begin{align}
	(d(\af)\bt+\af d(\bt))(x,y)&=(\af(x,y)\bt(y,y)+\af(x,x)\bt(x,y))L(e_{xy})(x,y)\label{(af(xy).bt(yy)+af(xx).bt(xy))L(e_xy)(xy)}\\
	&\quad+\sum_{x<z<y}\af(x,z)\bt(z,y)(L(e_{xz})(x,z)+L(e_{zy})(z,y))\label{sum-af(xz)bt(zy)(L(e_xz)(xz)+L(e_zy)(zy))}\\
	&\quad+\sum_{x<u\le z\le y}\af(u,z)\bt(z,y)L(e_{uz})(x,z)\label{sum-af(uz)bt(zy)L(e_uz)(xz)}\\
	&\quad+\sum_{x\le z\le v< y}\af(x,z)\bt(z,v)L(e_{zv})(z,y)\label{sum-af(xz)bt(zv)L(e_zv)(zy)}\\
	&\quad+\sum_{x\le v<z\le y}\alpha(x,v)\bt(z,y)L(e_{xv})(x,z)\label{sum-af(xv)bt(zy)L(e_xv)(xz)}\\
	&\quad+\sum_{x\le z<u\le y}\af(x,z)\bt(u,y)L(e_{uy})(z,y).\label{sum-af(xz)bt(uy)L(e_uy)(zy)}
	\end{align}
	
	The sum of \cref{sum-af(xz)bt(zy)(L(e_xz)(xz)+L(e_zy)(zy)),(af(xy).bt(yy)+af(xx).bt(xy))L(e_xy)(xy)} equals
	\begin{align}
	&\left(\alpha(x,y)\bt(y,y)+\af(x,x)\bt(x,y)+\sum_{x<z<y}\af(x,z)\bt(z,y)\right)L(e_{xy})(x,y)\notag\\
	&=(\af\bt)(x,y)L(e_{xy})(x,y)\label{(af.bt)(xy)L(e_xy)(xy)},
	\end{align}
	which is due to \cref{L(e_xy)(xy)+L(e_yz)(yz)=L(e_xz)(xz)}. 
	
	Now, 
	\begin{align*}
	L(e_{uz})(x,z)=L(e_u)(x,u)=L(e_{uy})(x,y)
	\end{align*}
	for all $x<u<z\le y$ and
	\begin{align*}
	L(e_{zy})(x,y)=L(e_z)(x,z)
	\end{align*}
	for all $x<z<y$ by \cref{L(e_xy)(uy)=L(e_x)(ux)}. So \cref{sum-af(uz)bt(zy)L(e_uz)(xz)} becomes 
	\begin{align}
	&\af(y,y)\bt(y,y)L(e_y)(x,y)+\sum_{x<z<y}\af(z,z)\bt(z,y)L(e_z)(x,z)\notag\\
	&\quad+\sum_{x<u<z\le y}\af(u,z)\bt(z,y)L(e_{uz})(x,z)\notag\\
	&=\af(y,y)\bt(y,y)L(e_y)(x,y)+\sum_{x<z<y}\af(z,z)\bt(z,y)L(e_{zy})(x,y)\notag\\
	&\quad+\sum_{x<u<z\le y}\af(u,z)\bt(z,y)L(e_{uy})(x,y)\notag\\
	&=\sum_{x<u\le z\le y}\af(u,z)\bt(z,y)L(e_{uy})(x,y)=\sum_{x<u\le y}(\af\bt)(u,y)L(e_{uy})(x,y)\label{sum-(af.bt)(uy)L(e_uy)(xy)}
	\end{align}
	Similarly, \cref{sum-af(xz)bt(zv)L(e_zv)(zy)} is equal to
	\begin{align}\label{sum-(af.bt)(xv)L(e_xv)(xy)}
	\sum_{x\le v< y}(\af\bt)(x,v)L(e_{xv})(x,y).
	\end{align}
	
	Adding \cref{sum-af(xv)bt(zy)L(e_xv)(xz),sum-af(xz)bt(uy)L(e_uy)(zy)} and making changes of variables in the both sums, we obtain
	\begin{align}\label{sum-af(xa).bt(by)(L(e_xa)(xb)+L(e_by)(ay))}
	\sum_{x\le a<b\le y}\alpha(x,a)\bt(b,y)(L(e_{xa})(x,b)+L(e_{by})(a,y)).
	\end{align}
	But 
	\begin{align*}
	L(e_{xa})(x,b)+L(e_{by})(a,y)=L(e_a)(a,b)+L(e_b)(a,b)=0
	\end{align*}
	for all $x<a<b<y$ by \cref{L(e_xy)(uy)=L(e_x)(ux),L(e_xy)(xv)=L(e_y)(yv),L(e_x)(xy)+L(e_y)(xy)=0}. Thus,  
	\cref{sum-af(xa).bt(by)(L(e_xa)(xb)+L(e_by)(ay))} is equal to
	\begin{align}
	&\alpha(x,x)\sum_{x<b\le y}\bt(b,y)(L(e_x)(x,b)+L(e_{by})(x,y))\notag\\
	&\quad+\bt(y,y)\sum_{x\le a<y}\alpha(x,a)(L(e_{xa})(x,y)+L(e_y)(a,y))\notag\\
	&=\alpha(x,x)\bt(y,y)(L(e_x)(x,y)+L(e_y)(x,y))\label{af(xx).bt(yy)(L(e_x)(xy)+L(e_y)(xy))}\\
	&\quad+\alpha(x,x)\sum_{x<b<y}\bt(b,y)(L(e_x)(x,b)+L(e_{by})(x,y))\label{af(xx).sum-bt(by)(L(e_x)(xb)+L(e_by)(xy))}\\
	&\quad+\bt(y,y)\alpha(x,x)(L(e_x)(x,y)+L(e_y)(x,y))\label{bt(yy).af(xx)(L(e_x)(xy)+L(e_y)(xy))}\\
	&\quad+\bt(y,y)\sum_{x<a<y}\alpha(x,a)(L(e_{xa})(x,y)+L(e_y)(a,y))\label{bt(yy).sum-af(xa)(L(e_xa)(xy)+L(e_y)(ay))}
	\end{align}
	In view of \cref{L(e_x)(xy)+L(e_y)(xy)=0}, we know that \cref{af(xx).bt(yy)(L(e_x)(xy)+L(e_y)(xy)),bt(yy).af(xx)(L(e_x)(xy)+L(e_y)(xy))} are zero. 
	Since
	\begin{align*}
		L(e_x)(x,b)+L(e_{by})(x,y)&=L(e_x)(x,b)+L(e_b)(x,b)=0,\\
		L(e_{xa})(x,y)+L(e_y)(a,y)&=L(e_a)(a,y)+L(e_y)(a,y)=0
	\end{align*}
	thanks to \cref{L(e_xy)(uy)=L(e_x)(ux),L(e_xy)(xv)=L(e_y)(yv),L(e_x)(xy)+L(e_y)(xy)=0},
	\cref{af(xx).sum-bt(by)(L(e_x)(xb)+L(e_by)(xy)),bt(yy).sum-af(xa)(L(e_xa)(xy)+L(e_y)(ay))} are zero as well. 
	It follows that  \cref{sum-af(xa).bt(by)(L(e_xa)(xb)+L(e_by)(ay))} is zero.
	
	Combining the result of the previous paragraph with \cref{(af.bt)(xy)L(e_xy)(xy),sum-(af.bt)(uy)L(e_uy)(xy),sum-(af.bt)(xv)L(e_xv)(xy)}, we conclude that
	\begin{align*}
		(d(\af)\bt+\af d(\bt))(x,y)&=(\af\bt)(x,y)L(e_{xy})(x,y)+\sum_{x<u\le y}(\af\bt)(u,y)L(e_{uy})(x,y)\\
		&\quad+\sum_{x\le v< y}(\af\bt)(x,v)L(e_{xv})(x,y)\\
		&=d((\af\bt)|_x^y)(x,y),
	\end{align*}
	the latter being $d(\af\bt)(x,y)$ in view of \cref{L(af)(xy)=L(restr-of-af)(xy)} applied to $d$, whence \cref{d(af.bt)=d(af).bt+af.d(bt)}.
\end{proof}

\begin{thrm}\label{L-is-D+F}
	Let $L$ be a Lie $n$-derivation of $FI(P,R)$. If $R$ is $(n-1)$-torsion free, then $L=d+\tau$, where 
	$d$ is a derivation of $FI(P,R)$ and $\tau$ is a central-valued linear mapping which annihilates all the $n$-th commutators.
\end{thrm}
\begin{proof}
 See \cref{tau-and-d-defn,tau-centr-map,d-derivation}.
\end{proof}

The following example shows that the assumption of $(n-1)$-torsion free in Theorem \ref{L-is-D+F} is necessary. 

\begin{exm}\label{L=id-is-Lie}
	Let $n\ge 3$, $P=\{1,2\}$ with $1<2$ and $\operatorname{char}R=n-1$. Then there are non-proper Lie $n$-derivations of $FI(P,R)$.
\end{exm}
\begin{proof}
	Define $L$ to be the identity mapping $FI(P,R)\to FI(P,R)$. Then 
	\begin{align*}
		L(p_n(\af_1,\dots,\af_n))=\sum_{i=1}^n p_n(\af_1,\dots,\af_{i-1},L(\af_i),\af_{i+1},\dots,\af_n)
	\end{align*}
	is equivalent to $(n-1)p_n(\af_1,\dots,\af_n)=0$, which is true under the assumption $\operatorname{char}R=n-1$. 
	
	Suppose that $L=d+\tau$ for some derivation $d$ and central-valued linear mapping $\tau$. Then
	\begin{align*}
		e_1=L(e_1)=d(e_1)+\tau(e_1).
	\end{align*}
	Since $\tau(e_1)\in Z(FI(P,R))$, we have $\tau(e_1)(1,1)=\tau(e_1)(2,2)$, whence
	\begin{align*}
		d(e_1)(1,1)=1-\tau(e_1)(1,1)=1-\tau(e_1)(2,2)\ne -\tau(e_1)(2,2)=d(e_1)(2,2).
	\end{align*}
	However,
	\begin{align*}
		d(e_1)=d(e_1^2)=e_1d(e_1)+d(e_1)e_1,
	\end{align*} 
	so $d(e_1)(2,2)=0$ and $d(e_1)(1,1)=2d(e_1)(1,1)$, which again implies that $d(e_1)(1,1)=0$, a contradiction.
\end{proof}

\begin{probl}
We would like to point out that all involved Lie-type derivations in our current work are linear.  It is natural to ask whether 
\cref{L-is-D+F} holds without the assumption of additivity. That is, we can investigate multiplicative 
Lie-type derivations of finitary incidence algebras, which is motivated by \cite{BenkovicEremita2, FosnerWeiXiao, WangWang}.
\end{probl}

\section*{Acknowledgements}
Mykola Khrypchenko was partially supported by the Foreign High-level Cultural and Educational Experts Project  
of the Beijing Institute of Technology.


\end{document}